\documentclass[]{amsart}
\usepackage{latexsym,xy}
\xyoption{all}
\usepackage{graphicx,amscd,amssymb,latexsym,subfigure,verbatim,hyperref,epsfig,supertabular}

\newtheorem{defn0}{Definition}[section]
\newtheorem{prop0}[defn0]{Proposition}
\newtheorem{conj0}[defn0]{Conjecture}
\newtheorem{thm0}[defn0]{Theorem}
\newtheorem{lem0}[defn0]{Lemma}
\newtheorem{corollary0}[defn0]{Corollary}
\newtheorem{example0}[defn0]{Example}
\newtheorem{remark0}[defn0]{Remark}
\newtheorem{que0}[defn0]{Question}

\newenvironment{defn}{\begin{defn0}}{\end{defn0}}

\newenvironment{thm}{\begin{thm0}}{\end{thm0}}
\newenvironment{lem}{\begin{lem0}}{\end{lem0}}

\newenvironment{cor}{\begin{corollary0}}{\end{corollary0}}
\newenvironment{exm}{\begin{example0}\rm}{\end{example0}}
\newenvironment{remark}{\begin{remark0}\rm}{\end{remark0}}
\newcommand{\V}{{\bf V }}
\renewcommand{\P}{{\mathbb{P}}}

\newcommand{\OPP}{{\mathcal{O}_{\mathbb{P}^1 \times \mathbb{P}^1}}}
\newcommand{\PP}{{\mathbb{P}^1 \times \mathbb{P}^1}}

\newcommand{\spn}{{\mathrm{Span}}}

\newcommand{\Sing}{{\mathrm{Sing}}}

\newcommand{\Syz}{{\mathrm{Syz}}}

\numberwithin{equation}{section}

\begin{document}
\title[Tensor product surfaces and linear syzygies]%
{Tensor product surfaces and linear syzygies} 

\author{Eliana Duarte}
\address{Department of Mathematics, University of Illinois, 
Urbana, IL 61801}
\email{emduart2@math.uiuc.edu}

\author{Hal Schenck}
\thanks{Schenck supported by NSF 1068754, NSF 1312071}
\address{Department of Mathematics, University of Illinois, 
Urbana, IL 61801}
\email{schenck@math.uiuc.edu}


\begin{abstract}
Let $U \subseteq H^0(\OPP(a,b))$ be a basepoint free 
four-dimensional vector space, with $a,b \ge 2$.
The sections corresponding to $U$ determine a regular map 
$\phi_U: \PP \longrightarrow \P^3$. We show that there can
be at most one linear syzygy on the associated
bigraded ideal $I_U \subseteq k[s,t;u,v]$. Existence of
a linear syzygy, coupled with the assumption that $U$ is
basepoint free, implies the existence of an additional 
``special pair'' of minimal first syzygies. Using results 
of Botbol \cite{bot}, we show that these three syzygies are 
sufficient to determine the implicit equation of $\phi_U(\PP)$,
and that $\Sing(\phi_U(\PP))$ contains a line.
\end{abstract}
\maketitle

\section{Introduction}
A tensor product surface is the image of a map $\PP \longrightarrow \P^3$.
Such surfaces arise in geometric modeling, and it is often
useful to find the implicit equation for the surface. Standard tools
such as Gr\"obner bases and resultants tend to be slow, and the best
current methods rely on Rees algebra techniques. The use of such methods 
was pioneered by the geometric modeling community (e.g. 
Sederberg-Chen \cite{sc}, Sederberg-Goldman-Du \cite{sgd},  
Sederberg-Saito-Qi-Klimaszewksi \cite{ssqk}, Cox-Goldman-Zhang \cite{cgz}). 
Further work on using Rees algebras in
implicitization appears in Bus\'e-Jouanolou \cite{bj}, 
Bus\'e-Chardin \cite{bc}, Botbol \cite{bot} and 
Botbol-Dickenstein-Dohm \cite{bdd}; see Cox \cite{cox2} for a nice overview.
A key tool is the approximation complex $\mathcal{Z}$, introduced by 
Herzog-Simis-Vasconcelos in \cite{hsv1}, \cite{hsv2}. 
\begin{defn} Let $I = \langle f_1, \ldots, f_n \rangle \subseteq R=k[x_1,\ldots x_m]$, and let $K_i \subseteq \Lambda^i(R^n)$ be the kernel of 
the $i^{th}$ Koszul differential on $\{f_1, \ldots, f_n\}$. 
The approximation complex $\mathcal{Z}$ is a complex 
of $S=R[y_1,\ldots, y_n]$ modules, with $i^{th}$ term $\mathcal{Z}_i = S \otimes_R K_i$, and differential the Koszul differential on $\{y_1, \ldots, y_n\}$.
\end{defn}
It follows from Definition 1.1 that $H_0(\mathcal{Z})$ is the 
symmetric algebra $S_I$ on $I$, and that $K_1$ is $\Syz(I)$. For 
a fixed degree $\mu$, the matrix representing the first 
differential $d^1$ in $\mathcal{Z}$ in degree $\mu$ is obtained 
by rewriting each syzygy on $I$
\[
\sum_{i=1}^n g_i e_i \mbox{ with } \sum_{i=1}^n g_i f_i=0
\]
as $\sum_{i=1}^n g_i y_i$, but in terms of a choice of 
basis for $R_\mu$, so that the entries of $d^1_\mu$ are 
elements of $k[y_1,\ldots, y_n]$. This generalizes to the bigraded setting. Let $R = k[s,t,u,v]$ be a bigraded ring 
over an algebraically closed field $k$, with $s,t$ of 
degree $(1,0)$ and $u,v$ of degree $(0,1)$. Note that the bidegree 
$(a,b)$ graded piece $R_{a,b}$ of $R$ corresponds exactly to the 
global sections $H^0(\OPP(a,b))$.
\begin{defn} 
Suppose $U \subseteq R_{a,b}$ has basis $\{ p_{0}, p_{1}, p_{2}, p_{3} \}$,
such that the $p_i$ have no common zeroes on $\PP$, and let $I_U$ denote 
the ideal 
$\langle  p_{0}, p_{1}, p_{2}, p_{3}\rangle \subset R$. 
Since the $p_i$ have no common zeroes on $\PP$, they define a regular map 
$\phi_U: \PP \longrightarrow \P^3$, and we write $X_U$ for 
$\phi_U(\PP) \subseteq \P^3.$ 
\end{defn}
The assumption that $U$ is basepoint free means that 
$\sqrt{I_U} = \langle s,t\rangle \cap \langle u,v \rangle.$
In this setting, work of \cite{bdd} gives conditions on $\mu$ so
that the determinant of $d^1_\mu$ is a power of the implicit equation 
for $X_U$. Motivated by \cite{cds}, in \cite{ssv}, 
Schenck-Seceleanu-Validashti show that for 
tensor product surfaces of bidegree $(2,1)$, the existence of a linear
syzygy on $I_U$ imposes very strong conditions on $X_U$. We show 
this is not specific to the bidegree $(2,1)$ case. Our main result is:
\vskip .05in
\noindent{\bf Theorem}: If $a,b \ge 2$ and $U$ is basepoint free, then there
is at most one linear first syzygy on $I_U$. A linear first syzygy 
gives rise to a special pair of additional first syzygies.
These three syzygies determine the degree $(2a-1,b-1)$ component of 
the approximation complex $\mathcal{Z}$. By \cite{bot}, the determinant of the 
resulting square matrix is a power of the implicit equation of $X_U$. 

\begin{exm}\label{ex1}
Suppose $(a,b)=(2,2)$, and  
\[
U = \spn \{t^2u^2+s^2uv, t^2uv+s^2v^2, t^2v^2, s^2u^2 \} \subseteq H^0(\OPP(2,2)),
\]
which has a first syzygy of bidegree $(0,1)$. A computation shows that $I_U$ has 
seven minimal first syzygies, in bidegrees 
\[
(0,1),(2,1),(2,1),(0,3),(2,2),(4,1),(6,0).
\]
By Theorem~\ref{EXTRA}, the three syzygies of bidegree $(0,1),(2,1),(2,1)$ are generated by the columns of 
\[
\left[\begin{matrix}
v &0 &s^2u \\
-u & -t^2v &0 \\
 0 & t^2u+s^2v & 0\\
0 & 0 & -t^2u-s^2v 
\end{matrix}
 \right],
\]
and the bidegree $(2a-1,b-1) = (3,1)$ component of the first differential
in the approximation complex is 
\[
\left[
\begin{matrix}
x_0&0&0&0   &x_2&0&-x_3&0\\
-x_1&0&0&0  &0&0&x_0&0\\
0&x_0&0&0   &0&x_2&0&-x_3\\
0&-x_1&0&0  &0&0&0 &x_0\\
0&0&x_0&0   &-x_1&0&0&0\\
0&0&-x_1&0  &x_2&0&-x_3&0\\
0&0&0&x_0   &0&-x_1&0&0\\
0&0&0&-x_1  &0&x_2&0&-x_3
\end{matrix}
\right]
\]
The determinant of this matrix is 
\[
(x_0^3x_2+x_1^3x_3-x_0^2x_1^2)^2.
\]
By Corollary~\ref{EQN} this means the implicit equation defining $X_U$ is 
$x_0^3x_2+x_1^3x_3-x_0^2x_1^2$, and $\phi_U$ is $2:1$ by Lemma~\ref{BLEM2}. 
By Corollary~\ref{SING} the codimension one singular locus 
of $X_U$ contains $\V(x_0,x_1)$; in fact, 
in this case equality holds.
\end{exm}
\subsection{Algebraic tools}
Two results from previous work will be especially useful; for additional background on approximation complexes and bigraded commutative algebra, see \cite{ssv}.
\begin{lem}\label{LS1}\cite{ssv}
If $I_U$ has a linear first syzygy of bidegree $(0,1)$, then 
\[
I_U = \langle pu, pv, p_2,p_3 \rangle,
\]
where $p$ is homogeneous of bidegree $(a,b-1)$.
\end{lem}
A similar result holds if $I_U$ has a first syzygy of
degree $(1,0)$. The lemmas below (Lemmas 7.3 and 7.4 of Botbol \cite{bot}) also play
a key role. Botbol notes that the local cohomology module $(H_2)_{4a-1,3b-1}$ has 
dimension equal to the sum of the multiplicities at the basepoints, so if 
$U$ is basepoint free, this module vanishes.
\begin{lem}\label{BLEM1}\cite{bot}
Suppose $a\le b$. If $\nu = (2a-1,b-1)$, then the determinant of the $\nu$ strand of the
approximation complex is of degree $2ab-\dim(H_2)_{4a-1,3b-1}$.
\end{lem}
\begin{lem}\label{BLEM2}\cite{bot}
If $U$ has basepoints with multiplicities $e_x$, then
\[
\deg(\phi_U)\deg(F) = 2ab-\sum e_x, \mbox{ where } \langle F \rangle=I(X_U).
\]
\end{lem}
\noindent If $U$ is basepoint free, the determinant of the $\nu$ strand is the determinant of $(d^1)_\nu$.
\section{Proofs of main theorems}
\begin{thm}\label{MAIN}
 If $a,b \ge 2$ and $U$ is basepoint free, then there
can be at most one linear first syzygy on $I_U$.
\end{thm}
\begin{proof}
Suppose $L$ is a linear syzygy of bidegree $(0,1)$ on $I_U$. By Lemma~\ref{LS1}, we may assume 
\[
I_U = \langle pu, pv, p_2,p_3 \rangle =\langle p_0, p_1, p_2,p_3 \rangle,
\]
where $p$ is homogeneous of bidegree $(a,b-1)$. Suppose there is another minimal first linear syzygy of bidegree $(0,1)$ 
\[
\sum\limits_{i=0}^3 p_i\cdot(a_iu+b_iv) = 0.
\]
Let
\[
\begin{array}{ccc}
\widetilde{p_2} &= &\sum a_ip_i\\
\widetilde{p_3} &= & \sum b_ip_i,
\end{array}
\]
so $\widetilde{p_2}u+\widetilde{p_3}v = 0$. But the syzygy module on $[u,v]$ is
generated by $[v,-u]$, so we must have $\widetilde{p_2}=qv, \widetilde{p_3}=-qu$ for
some $q$ of bidegree $(a,b-1)$. If in addition  
\[
D = \det \left[\begin{matrix}
a_2 & a_3 \\
b_2 & b_3
\end{matrix}
 \right]\mbox{ is nonzero, then }
\] 
\[
I_U = \langle pu,pv, \widetilde{p_2},\widetilde{p_3} \rangle = \langle pu, pv, qu, qv \rangle.
\]
Example V.1.4.3 of \cite{h} shows that curves $\V(f)$ of bidegree $(a,b)$ and 
$\V(g)$ of bidegree $(c,d)$ on $\PP$ sharing no common component 
meet in $ad+bc$ points. If $p$ and $q$ share a common factor, then
clearly $I_U$ is not basepoint free; if they do not share a common
factor, then $\V(p,q)$ consists of $2ab-2a$ points; since $a,b \ge 2$,
this again forces $I_U$ to have basepoints. The same argument works
if the additional syzygy is of bidegree $(1,0)$, save that in this
case since $q$ is of degree $(a-1,b)$, $\V(p,q)$ consists of $2ab-a-b+1$ 
points, and again $I_U$ is not basepoint free.

Next, suppose $D=0$. If $a_2=a_3=b_2=b_3=0$, then 
the second minimal first syzygy involves only $pu$ and $pv$. If the
syzygy is of bidegree $(0,1)$, then by Lemma~\ref{LS1}, 
$(pu,pv)=(qv,qu)$. Thus
\[
pu=qv \Longrightarrow p=fv, q=fu \Longrightarrow fv^2=fu^2,
\]
a contradiction. If the syzygy is of bidegree 
$(1,0)$, then $(pu,pv)=(qs,qt)$, and 
\[
pu=qs \Longrightarrow p=fs, q=fu \Longrightarrow fsv=fut,
\]
again a contradiction.

Finally, if $D=0$ and $a_2,a_3,b_2,b_3$ are not all zero, then $c\cdot[a_2,b_2] = [a_3,b_3]$
for some $c \ne 0$, so letting $\widetilde{p_2} = p_2+cp_3$, we may assume the
syzygy involves only $pu,pv, \widetilde{p_2}$. If the syzygy is of degree 
$(0,1)$, letting $l_i=a_iu+b_iv$ for $i=0,1,2$, we have 
\[
pul_0+ pvl_1+ \widetilde{p_2}l_2=0.
\]
Since $\langle l_2 \rangle$ is prime, either $l_2 | ul_0+ vl_1$ or
$l_2 | p$. In the former case, $ul_0+ vl_1 = l_2l_3$ for some $l_3
\in k[u,v]_1$ hence $pl_3 + \widetilde{p_2}=0$. In particular 
$p | \widetilde{p_2}$, so $\V(p,p_3)$ contains $2ab-a$ points 
and $I_U$ is not basepoint free. 

In the latter case, 
$p'l_2 =p$ for some $p' \in R_{(a-2,b)}$, 
so $p'l_2ul_0+ p'l_2vl_1+ \widetilde{p_2}l_2=0$.
Hence $p'ul_0+ p'vl_1+ \widetilde{p_2}=0$, so $p'$ is a common factor of 
$p$ and $\widetilde{p_2}$ of bidegree $(a,b-2)$, so $\V(p',p_3)$ 
contains $2ab-2a$ points and $I_U$ is not basepoint free. A similar  
argument works if the additional syzygy is of bidegree $(1,0)$.
\end{proof}
\begin{thm}\label{EXTRA}
If $U$ is basepoint free, $a,b \ge 2$ and there is a linear syzygy
$L$ of bidegree $(0,1)$ on $I_U$, then there are two additional first 
syzygies $S_1,S_2$ of bidegree $(a,b-1)$, such that
\[
\dim  \langle L,S_1,S_2\rangle_{(2a-1,b-1)} = 2ab.
\]
\end{thm}
\begin{proof}
By Lemma~\ref{LS1} we may assume $(p_0,p_1)=(pu,pv)$. Write 
$p_2 = g_2v+f_2u$. Then $f_2p_0+g_2p_1-pp_2 =0$, so the kernel of 
$[pu, pv, p_2]$ contains the columns of the matrix 
\[
M = \left[\begin{matrix}
v & f_2 \\
-u & g_2 \\
 0 & -p
\end{matrix}
 \right].
\]
In fact, $M$ is the syzygy matrix of $[pu, pv, p_2]$: the sequence $\{pu, p_2\}$ is not regular iff the two polynomials share a common factor. If $u|p_2$, then let $p_2'= p_2+pv$; $u|p_2'$ or $p|p_2'$ imply $I_U$ is not basepoint free. So the depth of the ideal of $2 \times 2$ minors of $M$ is two and exactness follows from the Buchsbaum-Eisenbud criterion \cite{ebig}. Writing $p_3=f_3u+g_3v$, the syzygy module of $I_U$ contains the columns of $N=\spn \{L,S_1,S_2\},$ where 
\[
N = \left[\begin{matrix}
v & f_2 & f_3\\
-u & g_2 &g_3\\
 0 & -p  &0 \\
0 & 0 & -p
\end{matrix}
 \right].
\]
As the bottom $3 \times 3$ submatrix of $N$ is upper triangular, 
$\{L,S_1,S_2\}$ span a free $R$-module. The linear syzygy $L$ is
of bidegree $(0,1)$, so in the degree $\nu$ strand of the 
approximation complex it gives rise to  
\[
h^0(\OPP(2a-1,b-2)) = 2a(b-1)
\]
columns of the matrix of the first differential $d^1$.
The two syzygies $S_1,S_2$ of bidegree $(a,b-1)$ each give rise to 
\[
h^0(\OPP(a-1,0)) = a
\]
columns of the matrix of $d^1$. That the columns are independent follows
from the fact that $\{L,S_1,S_2\}$ span a free $R$-module. 
Hence, these syzygies yield $2ab$ columns the degree $\nu$ component 
of the matrix of $d^1$.
\end{proof}
For Theorem~\ref{MAIN} and Theorem~\ref{EXTRA} to hold, we need 
$a,b \ge 2$, even if $U$ is basepoint free. If either $a$ or $b$ is 
at most one, there can be additional minimal linear syzygies. For
example, if $(a,b)=(1,1)$, then there are four minimal linear 
first syzygies. However, it is easy to 
see that the theorems both hold if $L$ is of bidegree $(1,0)$.
\begin{cor}\label{EQN}
If $a,b \ge 2$, $U$ is basepoint free, and $I_U$ has a linear 
first syzygy, then the determinant of the degree $\nu = (2a-1,b-1)$ submatrix
of the first differential in the approximation complex is determined
by $\{L, S_1,S_2\}$.
\end{cor}
\begin{proof}
This follows from Lemma~\ref{BLEM1}, Lemma~\ref{BLEM2}, the 
remarks preceding those lemmas, and Theorem~\ref{EXTRA}.
\end{proof}
\begin{cor}\label{SING}
If $a,b \ge 2$, $U$ is basepoint free, and $I_U$ has a linear 
first syzygy, then the singular locus of $X_U$ contains
a line.
\end{cor}
\begin{proof}
Let $I_U = \langle pu,pv,p_2,p_3\rangle$. 
By Corollary~\ref{EQN}, the matrix representing the 
degree $\nu$ component $d^1$ has as its leftmost $2a(b-1)$ 
columns a block matrix $P$. For each monomial $m_c = s^{2a-1-c}t^c$
with $c \in \{0,\ldots,2a-1\}$, there is a $b \times b-1$ block $B$ 
corresponding to elements $m_c\cdot\{v^{b-2},\ldots, u^{b-2}\} \cdot L$, with
$L = vx_0-ux_1$, hence
\[
B=\left[\begin{matrix}
x_0    & 0    &\hdots & \hdots &0\\
-x_1   &x_0   &0      & \vdots  & 0\\
\vdots &-x_1  &\ddots & \vdots  &0  \\
\vdots & 0   &x_0      &\ddots  &0  \\
\vdots & \vdots &\vdots&\vdots &0  \\
0 & 0           & 0    & -x_1  &x_0\\
0 & 0           & 0    & 0     &-x_1
\end{matrix}
\right], \mbox{ and }P=\left[\begin{matrix}
B    & 0    & \hdots &0\\
0    &B    &\ddots  & 0\\
0    &0  &\ddots & 0  \\
0 & \hdots          & 0    & B
\end{matrix}
\right].
\]
Computing the Laplace expansion of the determinant using the $2ab-2a$ minors
of $P$ shows the implicit equation for $X_U$ takes the form
\[
x_0^{2ab-2a}\cdot f_0 + x_0^{2ab-2a-1}x_1 \cdot f_1 + \cdots + x_1^{2ab-2a} \cdot f_{2ab-2a}.
\]
So $X_U$ is singular along $\V(x_0,x_1)$, with multiplicity at 
least $2ab-2a$.
\end{proof}
\begin{remark}The specific form of the implicit equation
given above means that it suffices to find the $f_i$, 
and speeds up the computation.
\end{remark}
\section{Application to the bidegree $(2,2)$ case}
We close with some examples in the bidegree $(2,2)$ case; 
without loss of generality we assume $I_U$ has a linear first syzygy of
bidegree $(0,1)$, so $I_U = \langle pu,pv,p_2,p_3\rangle$. Hence $p$ is 
of bidegree $(2,1)$. There are three possible factorizations for $p$:
\pagebreak

\begin{enumerate}
\item $p$ is irreducible.

\item $p$ is a product of an irreducible form of bidegree $(1,1)$, and a
form of bidegree $(1,0)$. So $p=ql$, where $q =a_0su+a_1sv+a_2tu+a_3tv$ and 
$l=b_0s+b_1t$. The locus of such forms is the image of the map 
\[
\P(H^0(\OPP(1,1))) \times \P(H^0(\OPP(1,0))) = \P^3 \times \P^1 \longrightarrow \P^5,
\]
$(a_0:a_1:a_2:a_3) \times (b_0:b_1) \mapsto 
(a_0b_0:a_0b_1+a_2b_0:a_2b_1:a_1b_0:a_1b_1+a_3b_0:a_3b_1)$, which
is a quartic hypersurface
\[
Q = \V({x}_{2}^{2} {x}_{3}^{2}-{x}_{1} {x}_{2} {x}_{3} {x}_{4}+{x}_{0} {x}_{2}
     {x}_{4}^{2}+{x}_{1}^{2} {x}_{3} {x}_{5}-2 {x}_{0} {x}_{2} {x}_{3} {x}_{5}-{x}_{0} {x}_{1} {x}_{4} {x}_{5}+{x}_{0}^{2} {x}_{5}^{2}).
\]
Note that $\Sigma_{2,1} \subseteq \V(Q)$.

\item $p$ is a product of three linear forms, two of bidegree $(1,0)$ and 
one of bidegree $(0,1)$. Then
identifying the coefficients of $p=a_0s^2u +a_1stu+a_2t^2u+a_3s^2v+a_4stv+a_5t^2v$ with a point of ${\mathbb P}^5$, 
such a decomposition corresponds to a point on the Segre variety $\Sigma_{2,1}$, whose ideal is defined by the two by two minors of 
\[
\left[\begin{matrix} x_0 & x_1 & x_2 \\ x_3 & x_4 &x_5 
\end{matrix}\right].
\]
\end{enumerate} 
Examples of possible bigraded betti tables for these three
cases appear below, where $p_2$ and $p_3$ are chosen generically.
When $p_2$ and $p_3$ are also nongeneric, there are many additional
possible types of betti table. It would be interesting to prove 
that the tables below are always the bigraded
betti tables for generic choices of $p_2$ and $p_3$, and to classify 
the bigraded resolutions which are possible in the $(2,2)$ case, and we 
are working on this. For brevity, we denote $R(a,b)$ by $(a,b)$. In all three 
cases, $X_U$ has degree $2ab=8$, in contrast to Example~\ref{ex1}.
\begin{exm}\label{genericCase}
Suppose $p \not\in \V(Q)$. After a change of coordinates, we may assume $p$ 
is the point $(1:0:0:0:0:1)$, which corresponds to $p=s^2u+t^2v$.
\[
0 \leftarrow I_U \longleftarrow (-2,-2)^4 \longleftarrow 
\begin{array}{c}
 (-2,-3)\\
\oplus \\
(-4,-3)^2\\
 \oplus \\
(-4,-4)\\
\oplus \\
(-3,-5)^2\\
\oplus \\
(-6,-3)\\
\oplus \\
(-8,-2)
 \end{array} 
\longleftarrow 
\begin{array}{c}
(-4,-5)^3\\
 \oplus \\
(-6,-4)^2\\
\oplus \\
(-8,-3)^2
 \end{array} 
\longleftarrow 
\begin{array}{c}
(-6,-5)\\
\oplus \\
(-8,-4)
 \end{array} 
\longleftarrow 0
\]
The reduced singular locus of $X_U$ consists of curves of degrees $1$, $2$, and $3$.
\end{exm}
\pagebreak

\begin{exm}\label{OnQ}
Suppose $p \in \V(Q) \setminus \Sigma_{2,1}$. 
After a change of coordinates, we may assume $p$ is the 
point $(1:2:1:1:1:0)$, which corresponds to $s^2u+2stu+t^2u+s^2v+stv$.
\[
0 \leftarrow I_U \longleftarrow (-2,-2)^4 \longleftarrow 
\begin{array}{c}
 (-2,-3)\\
\oplus \\
(-4,-3)^2\\
 \oplus \\
(-4,-4)\\
\oplus \\
(-3,-5)^2\\
\oplus \\
(-6,-3)\\
\oplus \\
(-7,-2)
 \end{array} 
\longleftarrow 
\begin{array}{c}
(-4,-5)^3\\
 \oplus \\
(-6,-4)^2\\
\oplus \\
(-7,-3)^2
 \end{array} 
\longleftarrow 
\begin{array}{c}
(-6,-5)\\
\oplus \\
(-7,-4)
 \end{array} 
\longleftarrow 0
\]
The reduced singular locus of $X_U$ consists of curves of degrees $1$, $1$, and $4$.
\end{exm}
\begin{exm}\label{genericCase}
Suppose $p \in \Sigma_{2,1}$. After a change of coordinates, we may assume $p$ is the 
point $(1:1:1:1:1:1)$, which corresponds to $s^2u+stu+t^2u+s^2v+stv+t^2v$.
\[
0 \leftarrow I_U \longleftarrow (-2,-2)^4 \longleftarrow 
\begin{array}{c}
 (-2,-3)\\
\oplus \\
(-4,-3)^2\\
 \oplus \\
(-4,-4)\\
\oplus \\
(-3,-5)^2\\
\oplus \\
(-6,-2)\\
 \end{array} 
\longleftarrow 
\begin{array}{c}
(-4,-5)^3\\
 \oplus \\
(-6,-4)\\
\oplus \\
(-6,-3)
 \end{array} 
\longleftarrow (-6,-5)
\longleftarrow 0
\]
The reduced singular locus of $X_U$ consists of curves of degrees $1$ and $4$.
\end{exm}

\noindent{\bf Acknowledgments} We thank an anonymous referee 
for a careful reading of the paper, and for helpful comments.
This work arose from a question asked by R. Vakil at the 2013
SIAM meeting on applied algebraic geometry, and we thank the
organizers of the session on toric geometry I. Soprunov and B. Nill. 
Evidence for this work was provided
by many computations done using {\tt Macaulay2}, by Dan
Grayson and Mike Stillman. {\tt Macaulay2} is freely available at 
\begin{verbatim}
http://www.math.uiuc.edu/Macaulay2/
\end{verbatim}
and scripts to perform the computations are available at 
\begin{verbatim}
http://www.math.uiuc.edu/~schenck/Syzscript
\end{verbatim}
\bibliographystyle{amsplain}

\end{document}